\documentclass{article}
\usepackage[utf8]{inputenc}
\usepackage{amsfonts}
\usepackage{amssymb}
\usepackage{amsmath}
\usepackage{amsthm}
\usepackage{commath}
\usepackage{xcolor}
\usepackage{bbold}
\usepackage{mathrsfs}
\usepackage{enumitem}
\usepackage{esint}
\usepackage{graphicx}
\usepackage{cancel}
\usepackage[noabbrev]{cleveref}
\usepackage{tikz-cd}
\usepackage{url}

\newcommand{\R}{\mathbb{R}}
\newcommand{\T}{\mathbb{T}}

\newcommand{\Id}{\textup{Id}}
\newcommand{\A}{\mathcal{A}}

\newcommand{\dV}{\,d\mathcal{V}}

\renewcommand{\del}{\partial}

\newcommand{\M}{\mathcal{M}}
\newcommand{\N}{\mathbb{N}}

\newcommand{\K}{\mathcal{K}}

\newcommand{\CH}{\mathbb{H}}
\newcommand{\Haus}{\mathcal{H}}
\newcommand{\E}{\mathcal{E}}

\newtheorem{theorem}{Theorem}
\newtheorem{definition}[theorem]{Definition}
\newtheorem{lemma}[theorem]{Lemma}

\newtheorem{remark}[theorem]{Remark}
\newtheorem{proposition}[theorem]{Proposition}
\newtheorem{example}[theorem]{Example}

\tikzset{every picture/.style={line width=0.75pt}} 

\numberwithin{theorem}{section}
 
\title{Fueter equations and the search for a higher dimensional Hamiltonian Floer theory I: analytical foundations and compactness}
\author{Luca Asselle and Ronen Brilleslijper}
\date{\today}

\begin{document}

\maketitle

\begin{abstract}
\noindent We study a Floer-theoretic approach to harmonic maps from the two-torus into non-flat K\"ahler manifolds. Building on the complex-regularized polysymplectic (CRPS) formalism of \cite{CRPS}, which provides a Hamiltonian description of harmonic maps for which the associated equations are elliptic, we analyze the compactness of the associated moduli spaces of Fueter maps.
For compact quotients $Q$ of complex hyperbolic space, we exploit the structure of the Biquard-Gauduchon hyperk\"ahler metric to prove relative compactness under suitable smallness assumptions on the Hamiltonian. In the flat case, we establish the necessary quantitative $L^\infty$-estimates and outline a perturbative strategy for the non-flat setting.
\end{abstract}

%

\section{Introduction}
Starting from Floer's work in the 1980s, pseudo-holomorphic curve techniques have played a prominent role in symplectic geometry, in particular in the study of periodic orbits of Hamiltonian systems on a symplectic manifold $(M,\omega)$. In this article, we focus on the case where $M=T^*Q$ is a cotangent bundle equipped with its standard symplectic structure. When the Hamiltonian equals the kinetic energy, the periodic orbits correspond to geodesics on $Q$. 
The key insight of Floer theory is that periodic orbits arise as critical points of the Hamiltonian action functional. Taking the $L^2$-gradient flow with respect to an $\omega$-compatible metric leads to the Floer equation
\begin{align*}
    \del_s u + J \del_t u = \nabla H(u),
\end{align*}
where $u:\R\times S^1\to T^*Q$. A crucial ingredient in the theory is the compactness of suitable moduli spaces of Floer trajectories.

Harmonic maps can be viewed as a higher-dimensional analogue of geodesics. It is therefore natural to ask whether Floer-type techniques can be developed to study their existence. More precisely, a smooth map $f:(M,h)\to (Q,g)$ between Riemannian manifolds is called \textit{harmonic} if it is a critical point of the Dirichlet energy
\[
\mathbb E(f) := \frac 12 \int_M |\mathrm d f|^2 \, \mathrm d \mu_h,
\]
where $\mu_h$ denotes the volume form on $M$ induced by $h$. 
The theory of harmonic maps was initiated by Eells and Sampson \cite{Eells}, who introduced the harmonic map heat flow and showed that, under suitable curvature assumptions, arbitrary maps can be deformed into harmonic maps. Since then, harmonic maps and their associated heat flow have become central objects in geometric analysis.
Of particular interest is the case where $M=\Sigma$ is a surface, corresponding to the Sobolev borderline case for the Dirichlet energy. In this dimension, compactness phenomena become especially delicate due to energy concentration and bubbling. A cornerstone of the theory in this setting is the work of Sacks and Uhlenbeck \cite{sacks1981existence}. They introduced a perturbation of the Dirichlet energy (the so-called $\alpha$-energy), defined on a Banach manifold and enjoying improved compactness properties. By analyzing the limit as $\alpha \downarrow 1$, they established the existence of harmonic two-spheres and developed a detailed bubbling analysis. Their approach has proved extremely influential and has been adapted to numerous other geometric variational problems.

Although the theory of harmonic maps includes the theory of geodesics as a special case, the study of harmonic maps is substantially more difficult. One reason is that the natural domain of definition of the Dirichlet energy, namely $H^1(\Sigma,Q)$, does not carry the structure of a Hilbert manifold unless $Q$ is a quotient of $\R^n$. This makes a Morse-theoretic analysis highly delicate. For this reason, it would be highly desirable to develop Floer-type techniques applicable to harmonic maps. At present, however, very little is known in this direction.
The poly- and multisymplectic frameworks (see, e.g., \cite{gunther1987polysymplectic,deLeonField}) provide a Hamiltonian description of harmonic maps, but as shown in \cite{CRPS}, the resulting equations are not elliptic and therefore unsuitable for Floer-theoretic methods. To overcome this issue, \cite{CRPS} introduced a formalism that yields a Hamiltonian formulation of harmonic maps from $\T^2$ into any K\"ahler manifold for which the associated equations are elliptic. In that setting, a lower bound on the number of solutions $\T^2\to \T^{2n}$ to the non-linear Laplace equation was established using Floer-theoretic techniques. The aim of the present article is to take the first steps toward extending these methods to non-flat targets, focusing mainly on compactness properties of the relevant moduli spaces of Floer maps.
An alternative approach, inspired by \cite{sacks1981existence} and based on Morse homology for functionals on Banach manifolds, has recently been initiated by the first author and collaborators in \cite{Asselle,Asselle2}.

The formalism developed in \cite{CRPS} is called \textit{complex-regularized polysymplectic} (CRPS) geometry. In \Cref{sec:CRPS}, we will recap the important notions. The idea is that on a CRPS manifold $W$, we have two symplectic forms $\omega_1$ and $\omega_2$ so that we can formulate a higher-dimensional analogue of Hamilton's equation
\begin{align}\label{eq:polysympl1}
\omega_1(\cdot,\del_xZ)+\omega_2(\cdot,\del_yZ)=dH(Z),
\end{align}
where $Z:\T^2\to W$. In Proposition \ref{prop:harmonic}, we show that for a K\"ahler manifold $(Q,g)$ and Hamiltonian $H(q,p)=\frac{1}{2}|p|_g^2$ on $T^*Q$, this equation describes harmonic maps $\T^2\to Q$. Solutions to \Cref{eq:polysympl1} arise as critical points of an action functional, just like in the symplectic case, which is given by
\begin{align*}
    \A_H(Z)=\int_{\T^2}\left(\theta_1(\del_xZ)+\theta_2(\del_yZ)-H(Z)\right)\,dV,
\end{align*}
where $\theta_1$ and $\theta_2$ are primitives of $\omega_1$ and $\omega_2$ respectively.

\begin{remark}
     \Cref{eq:polysympl1} only makes sense when the domain of $Z$ is a parallelizable surface. For arbitrary Riemann surfaces, the corresponding formalism was developed in \cite{CRMS}. In this article, we only consider maps from $\T^2$, but we expect similar results to hold in the more general setting.
\end{remark}

In our setting, the role of pseudo-holomorphic curves will be taken over by pseudo-Fueter maps $Z:\R\times\T^2\to T^*Q$
\[
\del_s Z+J(Z)\del_x Z+K(Z)\del_yZ=\nabla H(Z),
\]
where $J$ and $K$ are anti-commuting almost complex structures compatible with the CRPS structure on $T^*Q$. In general, these almost complex structures need not be integrable, and the compactness theory for the corresponding moduli spaces of pseudo-Fueter maps is very complex and not yet available. In this article, we will restrict to a suitable neighborhood of the zero-section in $T^*Q$, where we can assume that $J$ and $K$ are integrable (see \Cref{thm:Feix}). This allows us to use the methods of \cite{walpuski2017compactness} to prove compactness results for the relevant moduli space $\M$ of Fueter maps. 
\begin{theorem}\label{thm:mainintro}
    Let $Q$ be a compact quotient of complex hyperbolic space by a group of isometries and let $\mathbb D^*_{\delta_0}Q\subseteq T^*Q$ be a hyperk\"ahler neighborhood of the zero section. Then,
    there exists $0<\delta_*<\delta_0$ such that 
     the moduli space $\M$ of Fueter maps defined in \Cref{sec:hyperkahler} is relatively compact in the $C^\infty_\mathrm{loc}$-topology, provided the following two conditions are satisfied: 
     \begin{enumerate}
         \item[a)] All elements in $\M$ have image contained in $\mathbb D^*_\delta Q$, for some $0<\delta<\delta_*$. 
         \item[b)] The Hamiltonian is sufficiently $C^1$-close to the kinetic Hamiltonian.
         \end{enumerate}
\end{theorem}

The reason to restrict to quotients of complex hyperbolic space, and more generally to locally symmetric Hermitian spaces, is that the Biquard-Gauduchon hyperk\"ahler metric on the cotangent bundle is essentially explicit and block-diagonal with respect to the horizontal-vertical splitting induced by the K\"ahler metric on the base (see \cite{biquard2021hyperkahler}). This makes several computations significantly more transparent. We expect the theorem to hold also for general real-analytic K\"ahler manifolds. 

The proof of \Cref{thm:mainintro} is significantly more involved than in the flat case of \cite{CRPS}. One of the reasons is that the Heinz trick used in \cite{CRPS} cannot be applied directly. In our analysis, we need to make use of Walpuski's compactness theory for Fueter sections (\cite{walpuski2017compactness}). 
Condition $a)$ in \Cref{thm:mainintro} can be interpreted as a quantitative $L^\infty$-estimate. We expect such a bound to hold whenever the Hamiltonian is $C^1$-close to the kinetic Hamiltonian. In \Cref{sec:bounds}, we prove this for flat targets and provide indications for how to tackle this for non-flat targets. We emphasize that the quantitative confinement is only needed to exclude the residual singular set arising in Walpuski’s analysis, whereas the bubbling locus can already be ruled out under qualitative confinement assumptions. A finer understanding of harmonic spheres for the Biquard–Gauduchon metric, specifically of the energy threshold function governing their appearance, would likely allow one to relax such a condition. This suggests an interesting and largely unexplored geometric direction.

We will start this article with a recap of the CRPS formalism from \cite{CRPS}. Then in \Cref{sec:hyperkahler} we restrict ourselves to the hyperk\"ahler neighborhood and define the relevant moduli space of Floer maps. \Cref{sec:compactness} recalls Walpuski's compactness result and proves the energy bound needed to apply it. Next, in \Cref{sec:bubbles,sec:sing} we inevstigate possible sources of non-compactness and exclude those. Finally, \Cref{sec:bounds} comments on the necessary $L^\infty$-bounds

\textbf{Acknowledgements:} The authors would like to thank Oliver Fabert and Thomas Walpuski for helpful discussions. L.A. is partially supported by the DFG-grant 540462524 “Morse theoretical methods in Analysis, Dynamics and Geometry”.

\section{CRPS geometry}\label{sec:CRPS}
In this section, we give a brief recap of the complex regularized polysymplectic formalism (CRPS for short) introduced in \cite{CRPS}. We begin with a definition.
\begin{definition}
    Let $(W,I)$ be a complex manifold. A pair of symplectic forms $\Omega=(\omega_1,\omega_2)$, for $\omega_1,\omega_2\in\Omega^2(W)$ is called a \emph{CRPS form} if $\omega_2=-\omega_1(\cdot,I\cdot)$. The triple $(W,I,\Omega)$ is called a \emph{CRPS manifold}.
\end{definition}
The main example is given by cotangent bundles of complex manifolds.
\begin{example}\label{ex:cotangent}
    Let $(Q,i)$ be a complex manifold and $W=T^*Q$ its cotangent bundle equipped with the induced complex structure $I$. The Liouville 1-form $\lambda$ is defined by $\lambda(Z)=p\circ d\pi(q)$, for $Z=(q,p)\in W$ and $\pi:W\to Q$ the projection. The canonical CRPS form on $W$ is given by $\omega_1=d\lambda$ and $\omega_2=-d(\lambda\circ I)$.
\end{example}

It was also proven in \cite{CRPS} that locally all CRPS manifolds are standard.
\begin{theorem}\label{thm:Darboux}
    For any CRPS manifold $(W,I,\Omega=(\omega_1,\omega_2))$ and any $Z\in W$ there exist local coordinates $(q^j_1,q^j_2,p_j^1,p_j^2)_{j=1,\dots,n}$ around $Z$ such that locally
    \begin{align*}
        I\frac{\del}{\del q^j_1}=\frac{\del}{\del q^j_2} && I\frac{\del}{\del p_j^1}=-\frac{\del}{\del p_j^2}
    \end{align*}                                                                                       
    and
    \begin{align*}
        \omega_1=dp_j^1\wedge dq^j_1+dp_j^2\wedge dq^j_2 && \omega_2= dp_j^1\wedge dq^j_2-dp_j^2\wedge dq^j_1,
    \end{align*}
    where we use Einstein's summation convention on upper and lower indices $j$.
\end{theorem}

A main motivation for the CRPS formalism is that it provides a natural Hamiltonian framework for certain classes
of PDEs for maps from the two-torus $\T^2$ into a target manifold. More precisely, the pair
$\Omega=(\omega_1,\omega_2)$ plays the role of a ``polysymplectic'' structure with two distinguished directions
(on $\T^2$ these will be $x$ and $y$), and suitable choices of Hamiltonians $H$ give rise to first-order systems
for maps $Z:\T^2\to W$ obtained as Euler--Lagrange equations of an action functional.
In particular, this includes the harmonic map equation for maps $u:\T^2\to Q$ when $Q$ is a K\"ahler manifold.
Indeed, in Proposition~\ref{prop:harmonic} we show that the harmonic map equation can be written as a CRPS
Hamiltonian equation on the cotangent bundle $W=T^*Q$ endowed with its canonical CRPS structure from
Example~\ref{ex:cotangent}, choosing as Hamiltonian the kinetic energy associated with the K\"ahler metric $g$.
From now on, let $(W,I,\Omega=(\omega_1,\omega_2))$ be an \emph{exact} CRPS manifold, so that $\omega_i=d\theta_i$
for $i=1,2$. We study smooth maps $Z:\T^2\to W$, where $(x,y)$ denote the standard coordinates on $\T^2$ and
$dV$ is a fixed volume form (for instance $dV=dx\wedge dy$). In analogy with the classical symplectic action,
we consider the CRPS action functional\footnote{As in symplectic geometry, when the CRPS form is non-exact one can still define an action functional on suitable classes of maps (e.g.\ null-homotopic maps) by choosing capping data. In this article we will only
work in the exact setting.}
\begin{align*}
    \A(Z) = \int_{\T^2}\left(\theta_1(\del_xZ)+\theta_2(\del_yZ)\right)\, \mathrm d\mathcal V,
\end{align*} 
and, given a Hamiltonian $H:\T^2\times W\to\R$, we define the \textit{Hamiltonian action}
\begin{align*}
    \A_H(Z)=\A(Z)-\int_{\T^2}H(Z)\,\mathrm d\mathcal V.
\end{align*}
A standard variational argument shows that critical points of $\A_H$ are precisely the maps $Z$ solving the PDE 
\begin{align}\label{eq:polysympl}
    \omega_1(\cdot,\del_xZ)+\omega_2(\cdot,\del_yZ)=\mathrm dH(Z).
\end{align}

\begin{proposition}\label{prop:harmonic}
    Let $(Q,i,g)$ be a K\"ahler manifold and take the canonical CRPS structure on $T^*Q$ as defined in \Cref{ex:cotangent}. Define $H^0(q,p)=\frac{1}{2}|p|^2$ on $T^*Q$, where the norm is taken with respect to the dual metric of $g$. Then critical points $Z:\T^2\to T^*Q$ of $\A_{H^0}$ project to harmonic maps under $\pi$. That is, if $q=\pi\circ Z:\T^2\to Q$, then $q$ is harmonic.
\end{proposition}
\begin{proof}
    We may prove this statement in local coordinates. Pick holomorphic local coordinates $q^j=q^j_1+iq^j_2$ for $j=1,\ldots,n$ on $Q$ and induced holomorphic local coordinates $p_j=p_j^1-ip_j^2$. Note the minus sign in the definition of $p_j$, coming from the fact that $i^*=-i$. In these local coordinates we may write 
    \[ g=h_{j\bar{k}}\mathrm dq^j\otimes \mathrm d\bar{q}^k+h_{\bar{j}k}\mathrm d\bar{q}^j\otimes \mathrm dq^k,\]
    where 
    \begin{align*}
        h_{j\bar{k}}=\overline{h_{\bar{j}{k}}} && h_{\bar{k}j}=h_{j\bar{k}} && h_{jk}=h_{\bar{j}\bar{k}}=0.
    \end{align*}
    We denote the inverse metric with upper indices, where we use the convention $h^{j\bar{l}}h_{k{\bar{l}}}=\delta^j_k$. The condition that $g$ is a K\"ahler metric, translates to (see \cite{jost2005riemannian})
    \begin{align}\label{eq:Kahlercondition}
        \frac{\del h_{j\bar{l}}}{\del \bar{q}^k}=\frac{\del h_{j\bar{k}}}{\del\bar{q}^l}.
    \end{align}
    This implies that the only non-vanishing Christoffel symbols are $\Gamma^l_{jk}=h^{l\bar{r}}\frac{\del h_{j\bar{r}}}{\del q^k}$ and their conjugates $\Gamma^{\bar{l}}_{\bar{j}\bar{k}}$.

    The local expression of Equation \ref{eq:polysympl} in these holomorphic coordinates is
    \begin{align*}
        -\del_{\bar{t}}\bar{p}_j&=\frac{\del H}{\del \bar{q}^j}\\
        \del_t q^j&=\frac{\del H}{\del p_j},
    \end{align*}
    where $\del_t=\frac{1}{2}(\del_x-i\del_y)$ and $\del_{\bar{t}}=\frac{1}{2}(\del_x+i\del_y)$.
    Now, since $H$ is given by $$H(q,p)=\frac{1}{2}h^{j\bar{k}}p_j\bar{p}_k+\frac{1}{2}h^{\bar{j}k}\bar{p}_jp_k,$$ we get
    \begin{align}
        -\del_{\bar{t}}\bar{p}_j&=\frac{1}{2}\frac{\del h^{l\bar{k}}}{\del \bar{q}^j}p_l\bar{p}_k+\frac{1}{2}\frac{\del h^{\bar{l}k}}{\del \bar{q}^j}\bar{p}_l p_k\label{eq:pder}\\
        \del_t q^j&=h^{j\bar{k}}\bar{p}_k\label{eq:qder}.
    \end{align}
    The second equation yields $\bar{p}_l=h_{j\bar{l}}\del_tq^j$, which after taking the holomorphic derivative yields
    \begin{align}\label{eq:LHS}
        -\del_{\bar{t}}\bar{p}_l=-\frac{\del h_{j\bar{l}}}{\del q^k}\del_{\bar{t}}q^k\del_t q^j-\frac{\del h_{j\bar{l}}}{\del \bar{q}^k}\del_{\bar{t}}\bar{q}^k\del_tq^j-h_{j\bar{l}}\del_{\bar{t}}\del_tq^j.
    \end{align}
    By taking the derivative of $h^{j\bar{l}}h_{k{\bar{l}}}=\delta^j_k$ it follows that $\frac{\del h^{r\bar{k}}}{\del \bar{q}^m}=-\frac{\del h_{j\bar{l}}}{\del \bar{q}^m}h^{r\bar{l}}h^{j\bar{k}}$. Filling this in in \cref{eq:pder} and combinging with the expression for $p$ coming from \cref{eq:qder}, we get that
    \begin{align}\label{eq:RHS}
        -\del_{\bar{t}}\bar{p}_m=-\frac{1}{2}\frac{\del h_{j\bar{l}}}{\del \bar{q}^m}\del_{\bar{t}}\bar{q}^l\del_t q^j-\frac{1}{2}\frac{\del h_{\bar{j}l}}{\del\bar{q}^m}\del_tq^l\del_{\bar{t}}\bar{q}^j.
    \end{align}
    Now, equating \cref{eq:LHS} and \cref{eq:RHS}, multiplying by the inverse metric and renaming indices yields
    \begin{align}
        \nonumber-\del_{\bar{t}}\del_tq^l &= h^{\bar{m}l}\frac{\del h_{j\bar{m}}}{\del q^k}\del_{\bar{t}}q^k\del_tq^j+h^{\bar{m}l}\left(\frac{\del h_{j\bar{m}}}{\del \bar{q}^k}-\frac{\del h_{j\bar{k}}}{\del\bar{q}^m}\right)\del_{\bar{t}}\bar{q}^k\del_tq^j\\
        &=\Gamma^l_{jk}\del_{\bar{t}}q^k\del_t q^j.\label{eq:harmonictorus}
    \end{align}
    Here, the second term on the right-hand side of the first equality vanishes by \cref{eq:Kahlercondition}. \Cref{eq:harmonictorus} precisely gives the condition for $q:\T^2\to Q$ being a harmonic map (see \cite{jost2005riemannian}).
\end{proof}

In \cite{CRPS}, the CRPS setup was used to prove, among other things, a version of Arnold's conjecture for flat targets.
\begin{theorem}[Theorem 7.2 from \cite{CRPS}]\label{thm:Arnold}
    Put the standard CRPS structure on $T^*\T^{2n}$, where $\T^{2n}$ is equipped with the standard complex structure. For $H:\T^2\times T^*\T^{2n}\to \R$ defined by $H(t,q,p)=\frac{1}{2}|p|^2+h(t,q,p)$, there are at least $(2n+1)$ solutions $Z:\T^2\to T^*\T^{2n}$ to \cref{eq:polysympl}, provided $h$ is a smooth function with finite $C^2$-norm.
\end{theorem}

Theorem \ref{thm:Arnold} provides in particular a lower bound on the number of solutions $q:\T^2\to\T^{2n}$ to the nonlinear Laplace equation $-\Delta q = \nabla V(q)$, where $V:Q\to\R$ smooth potential. The proof relies roughly speaking on the study of the moduli space $\mathcal M$ of finite energy $L^2$-gradient flow lines of the Hamiltonian action, which hereafter will be called \textit{Fueter/Floer maps}, which are asymptotic to the zero-section. The three main ingredients of the proof are:
\begin{itemize}
    \item \textbf{A priori $L^\infty$-estimates:} both solutions to \cref{eq:polysympl} and elements of $\mathcal M$ stay within a bounded region of $T^*\T^{2n}$.
    \item \textbf{Fredholm theory:} $\mathcal M$  is a manifold for generic choices of $h$.
    \item \textbf{Compactness:} $\mathcal M$ is relatively compact.
\end{itemize}

In this paper, we take the first steps towards the extension of Theorem \ref{thm:Arnold} to cotangent bundles of non-flat K\"ahler manifolds. Our main result will be the pre-compactness of the space $\mathcal M$ (see Theorem \ref{thm:compactness}). Also, in \cref{sec:bounds}, we will comment on the required a priori $L^\infty$-estimates. 


\section{Hyperk\"ahler metrics on cotangent bundles}\label{sec:hyperkahler}
From this point on, we fix a real-analytic K\"ahler manifold $(Q,i,g)$ with K\"ahler form $\omega=g(\cdot,,i\cdot)$. Following Example \ref{ex:cotangent}, the cotangent bundle $T^*Q$ comes equipped with a canonical integrable complex structure $I$ and CRPS forms $\omega_1$ and $\omega_2=-\omega_1(\cdot,I\cdot)$. The following is a theorem from Feix \cite{feix2001hyperkahler}.
\begin{theorem}[\cite{feix2001hyperkahler}]
\label{thm:Feix}
    Let $(Q,i,g)$ be a real-analytic K\"ahler manifold. Then, there exists $\delta_0>0$ and a hyperk\"ahler metric $G$ on the open $\delta_0$-disk cotangent bundle $\mathbb D^*_{\delta_0}Q\subseteq T^*Q$ compatible with the CRPS structure. That is, there exist integrable complex structures $J$ and $K$ on $\mathbb D^*_{\delta_0}Q$ that satisfy
    \begin{itemize}
        \item $\omega_1=G(\cdot,J\cdot)$ and $\omega_2=G(\cdot,K\cdot)$,
        \item $\omega_0:=G(\cdot,I\cdot)$ is closed, and
        \item $IJ=K$.
    \end{itemize}
    Moreover, $G|_Q=g$. 
\end{theorem}

\begin{remark}
The hyperk\"ahler structure is globally defined for $Q=\T^{2n}$ and $Q$ a symmetric space of compact type, see e.g. \cite{biquard2021hyperkahler}. In general, the hyperk\"ahler structure is not globally defined on $T^*Q$. For instance, if $Q=\mathbb H^2$ hyperbolic plane (or one of its compact quotients) with Riemannian metric of constant curvature $-1$, then $\delta_0=1$, and $G$ explodes when approaching $|p|_g =1$. 
\end{remark}
With respect to the $L^2$-metric induced from $G$, we get that the gradient of the Hamiltonian action $\A_H$ is given by
\[\nabla^{L^2} \A_H(Z)=J(Z)\del_xZ+K(Z)\del_yZ-\nabla^GH(Z),\]
for $Z:\T^2\to D^*_{\delta_0}Q$. Thus, after interpreting the negative gradient flow equation 
$$\del_s Z = - \nabla^{L^2}\A_H(Z),\quad  Z:\R\to C^\infty(\T^2, \mathbb D^*_{\delta_0}Q),$$ 
as a PDE for $Z:\R\times \T^2 \to \mathbb D^*_{\delta_0} Q,$ we get the \textit{Fueter/Floer equation}
\begin{align}\label{eq:Floer}
    \del_sZ+J(Z)\del_xZ+K(Z)\del_yZ=\nabla^GH(Z).
\end{align}
We now define the moduli space $\mathcal M$ of Fueter/Floer maps rigorously.
 Fix $\delta\in (0,\delta_0)$, and  
let $M\subseteq C^\infty(\T^2,\mathbb D^*_\delta Q)$ be the subset of null-homotopic maps. Define a family of cut-off functions $\beta_\tau:\R\to[0,1]$ such that 
\begin{itemize}
    \item $\beta_\tau(s)=0$ for $s\leq -1$ and $s\geq \tau+1$ for all $\tau\geq 0$,
    \item $\beta_\tau|_{[0,\tau]}\equiv 1$ for $\tau\geq 1$,
    \item $0\leq \beta_\tau'(s)\leq 2$ for $s\in(-1,0)$ and $0\geq \beta_\tau'(s)\geq -2$ for $s\in(\tau,\tau+1)$,
    \item $\lim_{\tau\to 0^+}\beta_\tau=0$ in the strong topology on $C^\infty$.
\end{itemize}
This family is used to modify the Hamiltonian $H$. To this end, define 
$$H^\tau:\R\times\T^2\times T^*Q\to\R,\quad H^\tau_{s}(t,q,p):=\frac{1}{2}|p|_g^2+\beta_\tau(s)h(t,q,p),$$ 
and the modified action functional $G_{\tau,s}=\A_{H^\tau_s}$ that interpolates between $\A_{H^0}$ and $\A_{H}$. We denote by $Q\subseteq M$ the set of constant maps into the zero section of $T^*Q$. 
Following \cite{CRPS}, we introduce the moduli space $\M$ of pairs $(\tau,Z)$, where $\tau\geq 0$ and $Z:\R\to M$, satisfying the following conditions:
    \begin{enumerate}
        \item The map $Z$ is a solution of \eqref{eq:Floer} with Hamiltonian $H^\tau$.
        \item $Z$ has finite energy, that is, 
        \[E(Z)=\int_{-\infty}^\infty\int_{\T^2}|\del_sZ(s,t)|^2\mathrm d\mathcal V \mathrm ds < +\infty.\]
        \item $|Z(s,\cdot)|_g\to 0$ as $s\to\pm \infty$.
    \end{enumerate}
A crucial step in the proof of Theorem \ref{thm:Arnold} is the fact that $\M$ is relatively compact in the $C^\infty_\text{loc}$-topology. Indeed, this ensures the existence of limiting Fueter/Floer maps for $\tau\to\infty$, which will converge to critical points of $\A_H$. Our main theorem is the compactness of $\M$ on a class of non-flat target manifolds. 
\begin{theorem}\label{thm:compactness}
    Let $Q$ be a compact quotient of the complex hyperbolic space by a group of isometries and let $\mathbb D^*_{\delta_0}Q\subseteq T^*Q$ be equipped with the hyperk\"ahler metric $G$ as in Theorem \ref{thm:Feix}. Then, there exist $0<\delta_*<\delta_0$ and $C_*>0$, such that $\M$ is relatively compact in the $C^\infty_\mathrm{loc}$-topology, whenever $0<\delta<\delta_*$ and $h$ is such that the following hold:  
    \begin{enumerate}[label=\alph*)]
        \item \label{cond:C^0} For all $(\tau,Z)\in\M$ we have that $\mathrm{im}(Z)\subseteq \mathbb D^*_\delta Q$, and 
        \item $\|h\|_{C^1}<C_*$.
    \end{enumerate}
\end{theorem}

Before moving further we make some comments:
\begin{itemize}
\item  Condition \textit{a)} can be interpreted as a \emph{quantitative \(L^\infty\)-estimate}: it requires a uniform a priori confinement of all Fueter/Floer trajectories to the disk bundle \(\mathbb D^*_\delta Q\). We expect that such a bound holds whenever \(\|h\|_{C^1}\) is small. In Chapter 7 we prove this statement in the flat case \(Q=\T^{2n}\). The torus case is simpler because, roughly speaking, in the Bochner inequality for \(|p|_g^2\), no curvature terms appear (indeed the relevant geometry is flat), whereas for non-flat targets curvature terms enter and the Fueter/Floer equation only provides \(L^2\)-control on them rather than pointwise bounds. Chapter 7 also contains more detailed comments on this.

\item The proof of Theorem \ref{thm:compactness} is substantially more involved than in the flat case considered in \cite{CRPS} and builds on Walpuski's compactness theory for Fueter sections, see \cite{walpuski2017compactness}. One key reason is that, in the flat case, the Bochner inequality for \(|\nabla Z|^2\) becomes \emph{subcritical}, which allows one to apply the Heinz trick (compare Remark 3.5 in \cite{walpuski2017compactness} and \cite{HNS}). This mechanism is no longer available for non-flat targets.

\item Concretely, our argument proceeds in two steps. First we show that Walpuski's bubbling set \(\Gamma\) is empty (see Section 5). Then, using Esfahani's result that \(\Gamma=\emptyset\) forces the remaining singular set to detect non-trivial tri-holomorphic maps for \(G\), see \cite{esfahani2023towards}, we conclude that the singular set is empty as well (Section 6).

\item We shall emphasize that the quantitative confinement given by Condition $a)$ is needed only to rule out the singular set. To show that $\Gamma$ is empty, it is sufficient to assume that all elements in $\mathcal M$ have image in a fixed (not necessarily small) compact subset of $\mathbb D^*_{\delta_0}Q$ (\textit{qualitative $L^\infty$-estimates}).

\item We expect the conclusion of Theorem \ref{thm:compactness} to hold for arbitrary real-analytic K\"ahler manifolds \(Q\), but at present we focus on compact quotients of the complex hyperbolic space, where the geometry is sufficiently rigid to make the above strategy effective. A further advantage of restricting to compact quotients of complex hyperbolic space (and, more generally, to symmetric or locally symmetric Hermitian spaces) is that the Feix hyperk\"ahler metric agrees with the Biquard-Gauduchon metric \cite{biquard2021hyperkahler}, which is essentially explicit. In particular, with respect to the horizontal/vertical splitting determined by the Levi-Civita connection on the base, the metric is block-diagonal. This structural simplification makes several computations considerably more transparent.

\item The hyperk\"ahler nature of \(G\) is essential in our proof. In particular, the monotonicity lemma in Walpuski's paper \cite{walpuski2017compactness} uses integrability of the complex structures: when \(G\) is not hyperk\"ahler, a compactness theory for the resulting pseudo-Fueter maps is not yet available, and the proof of the monotonicity lemma in~\cite{walpuski2017compactness} does not go through.
\end{itemize}

As mentioned above, the hyperk\"ahler metric on $\mathbb D^*_{\delta_0}Q$ is difficult to use in computations. Both the construction from \cite{feix2001hyperkahler} and the construction of the same metric by \cite{kaledin1997hyperkaehler} are very involved and it is not clear what $J,K$ and the gradients with respect to this metric look like. The easiest metric to work with on cotangent bundles is the \textit{Sasaki metric}, which is defined as follows. Recall that the Levi-Civita connection of $g$ on $Q$ splits $TT^*Q$ into a vertical and horizontal bundle. That is, $TT^*Q\cong H\oplus V$, where $H\cong TQ$ and $V\cong T^*Q$. The Sasaki metric on $T^*Q$ is defined as $\bar{g}=g\oplus g$ in this splitting, where we identify $T^*Q\cong TQ$ by $g$. Unfortunately, the Sasaki metric is not hyperk\"ahler unless $Q$ is flat. This fact makes it at present  unsuitable for the proof of compactness. When $Q$ is a locally symmetric Hermitian space the situation improves drastically. In this case, \cite{biquard2021hyperkahler} proved that the hyperk\"ahler metric from \cite{feix2001hyperkahler} respects the splitting into horizontal and vertical bundles induced by $g$. Moreover, an explicit formula for the metric on $\mathbb D^*_{\delta_0}Q$ is given, which we recall below. Here, we denote by $R_{\xi,\xi'}:T_qQ\to T_qQ$ the curvature of $Q$, given by $R_{\xi,\xi'}=\nabla_{[\xi,\xi']}-[\nabla_\xi,\nabla_{\xi'}]$.
\begin{theorem}[\cite{biquard2021hyperkahler}]\label{thm:BG}
    The metric $G$ at $(q,\xi)\in TQ\cong T^*Q$ is given by 
    \begin{align*}
        G_{(q,\xi)}(X,Y) = g(A_\xi X^H,Y^H) + g(A_\xi^{-1} X^V, Y^V),
    \end{align*}
    for $X=X^H\oplus X^V,Y=Y^H\oplus Y^V\in T_{(q,\xi)}TQ\cong H_q\oplus V_q$ and 
    \begin{align*}
        A_\xi &= \Id + iR_{i\phi(iR_{i\xi,\xi})\xi,\phi(iR_{i\xi,\xi})\xi}, \qquad \phi : = \left(\frac{\sqrt{1+w}-1}{w}\right)^{\frac{1}{2}}.
    \end{align*}
\end{theorem}
\begin{remark}
 $iR(i\xi,\xi)$ gives a self-adjoint endomorphism of $T_qQ$. Thus, $T_qQ$ decomposes into eigenspaces $E_\lambda$ for real eigenvalues $\lambda$ of $iR_{i\xi,\xi}$. By definition the operator $\phi(iR_{i\xi,\xi})$ acts on $E_\lambda$ as multiplication by $\phi(\lambda)$. When $Q$ is a quotient of complex hyperbolic space, the metric $G$ is well-defined at all $\xi$ where the modulus of all eigenvalues of $iR_{i\xi,\xi}$ is less than 1. This requirement determines $\delta_0$.
\end{remark}

To illustrate the behavior of the operators $A_\xi$, we will compute them for complex hyperbolic space. Let $\CH^{2n}$ denote the complex hyperbolic space of dimension $2n$. That is, $\CH^{2n}=\{q\in\R^{2n}\mid |q|<1\}$ with metric $g=\frac{4}{(1-|q|^2)^2} g_{\text{flat}}$. This metric has constant holomorphic sectional curvature equal to $-4$. 
\begin{proposition}\label{prop:AxiCH}
    For $Q=\CH^{2n}$ the operator $A_\xi$ from Theorem \ref{thm:BG} reads
    \begin{align*}
        A_\xi = \begin{cases}
            1-4|\xi|^2\phi(-4|\xi|^2)^2 & \text{on }\  \textup{span}\{\xi,i\xi\}\\
            1-2|\xi|^2\phi(-4|\xi|^2)^2 & \text{on }\  \textup{span}\{\xi,i\xi\}^\perp.
        \end{cases}
    \end{align*}
\end{proposition}
\begin{proof}
    As the holomorphic sectional curvature is constant, we may compute the curvature tensor following\footnote{Note the difference in sign convention with \cite{beldjilali2023class}.} \cite{beldjilali2023class}:
    \begin{align}\label{eq:beldjilali}
        R_{X,Y}v &=  g(Y,v)X-g(X,v)Y+g(iY,v)iX\\
        &-g(iX,v)iY+2g(X,iY)iv,
    \end{align}
    where $X,Y,v\in T_qQ$. We apply Equation \ref{eq:beldjilali} to $X=i\xi$ and $Y=\xi$ and obtain
    \begin{align*}
        R_{i\xi,\xi}v=2g(\xi,v)i\xi-2g(i\xi,v)\xi+2|\xi|^2iv. 
    \end{align*}
    It follows that $iR_{i\xi,\xi}$ acts as multiplication by $-4|\xi|^2$ on the plane generated by $\xi$ and $i\xi$, and as multiplication by $-2|\xi|^2$ on directions orthogonal to this plane. Thus, 
    \[\phi(iR_{i\xi,\xi})\xi=\phi(-4|\xi|^2)\xi.\]
    By linearity, this means that 
    \[iR_{i\phi(iR_{i\xi,\xi})\xi,\phi(iR_{i\xi,\xi})\xi}=\phi(-4|\xi|^2)^2iR_{i\xi,\xi}.\]
    The proof follows from filling in that $A_\xi=\Id+iR_{i\phi(iR_{i\xi,\xi})\xi,\phi(iR_{i\xi,\xi})\xi}$.
\end{proof}


\section{Compactness for Fueter curves}\label{sec:compactness}
The main ingredient for the proof of Theorem \ref{thm:compactness} is the following result from \cite{walpuski2017compactness}. We denote by $\Haus^d$ the $d$-dimensional Hausdorff measure on $\R\times\T^2$ and for $\K\subseteq\R\times\T^2$ define
\[\E_\K(Z):=\int_\K |\mathrm dZ|^2 \mathrm d\mathcal V\wedge \mathrm ds,\]
where $\dV$ is the volume element on $\T^2$.
\begin{theorem}[Theorem 1.9 from \cite{walpuski2017compactness}]\label{thm:walpsuki}
    Suppose $\K=[-\mu,\mu]\times\T^2\subseteq \R\times \T^2$. Let $Z^m:\K\to \mathbb D^*_\delta Q$ be a sequence of solutions of
    \begin{align}\label{eq:Fueter}
        \del_sZ^m+J(Z^m)\del_xZ^m+K(Z^m)\del_yZ^m=\nabla^GH^{\tau_m}_s(Z^m),
    \end{align}
    for $\tau_m\geq0$, such that $\E_\K(Z^m)$ is uniformly bounded. Then there exists a subsequence, which we will still denote by $Z^m$, a closed subset $S\subseteq \K$ with $\Haus^1(S)<\infty$, and a map $Z:\K\backslash S\to \mathbb D^*_\delta Q$ with the following properties:
    \begin{itemize}
        \item The map $Z$ satisfies
        \begin{align*}
            \del_sZ+J(Z)\del_xZ+K(Z)\del_yZ=\nabla^GH^\tau_s(Z),
        \end{align*}
        where $\tau=\lim \tau_m$ exists, and $H^\tau=H$ if $\lim \tau_m=+\infty$.
        \item The sequence $Z^m|_{\K\backslash S}$ converges to $Z$ in $C^\infty_{\mathrm{loc}}$.
        \item The set $S$ can be written as $S=\Gamma\cup \textup{sing}(Z)$, where 
        \[\textup{sing}(Z):=\left\{z\in \K\middle| \limsup_{\rho\to0^+}\frac{1}{\rho}\int_{B_\rho(z)}|\mathrm dZ|^2>0\right\}\]
        and $\Haus^1(\textup{sing}(Z))=0$. Moreover, at $\Haus^1$-almost every point $z\in\Gamma$ the tangent space $T_z\Gamma$ is well-defined and for each such $z$ there exists a non-trivial holomorphic sphere $\psi_z:S^2\to \mathbb D^*_\delta Q$ with respect to a complex structure $\mathcal{I}$ on $\mathbb D^*_\delta Q$ of the form $\mathcal{I}=c_0I+c_1J+c_2K$ for $c_0^2+c_1^2+c_2^2=1$.
        \item $\Gamma = \textup{supp}(\Theta\Haus^1\lfloor S)$, for $\Theta:S\to [0,\infty)$ upper-semicontinuous.
    \end{itemize}
\end{theorem}

As is clear from the assumptions in Theorem \ref{thm:walpsuki}, we need to prove a uniform energy bound on the elements of $\M$.
\begin{lemma}\label{lem:energybound}
    Let $Q$ be a compact quotient of complex hyperbolic space by a group of isometries, $(\tau,Z)\in\M$ and $\K=[-\mu,\mu]\times \T^2$ for some $\mu\geq 0$. Then 
    \begin{align*}
        \E_\K(Z) < C(\K)(\delta^2+ \|h\|_{C^1}^2 + \|h\|_{C^1}),
    \end{align*}
    for some constant $C(\K)$ depending only on the volume of $\K$.
\end{lemma}
\begin{proof}
    First of all, a standard computation shows that 
    \begin{align}\label{eq:Ebound}
        E(Z) = -\int_{-\infty}^\infty\int_{\T^2}\beta'_\tau(s)h(t,q,p)\mathrm d\mathcal V \mathrm ds\leq 2||h||_{\mathrm{Hofer}} \leq 4 \|h\|_{C^1},
    \end{align}
    where the inequality follows from the assumptions on $\beta_\tau$ (see \cite[Lemma 10.1]{CRPS} or \cite{albers2016cuplength} for the first equality). Now note that $Z$ satisfies 
    \[\del_sZ + J\del_IZ=\nabla^G H^\tau_s(Z),\]
    where $\del_I=\del_x-I\del_y$. We see 
    \begin{align*}
        \int_\K|\del_IZ|^2\mathrm d\mathcal V\wedge \mathrm ds&=\int_\K\left(|\del_xZ|^2+|\del_yZ|^2-2G(\del_xZ,I\del_yZ)\right)\mathrm d\mathcal V\wedge \mathrm ds\\
        &=\int_\K\left(|\del_xZ|^2+|\del_yZ|^2\right)\mathrm d\mathcal V\wedge \mathrm ds-2\int_\K \mathrm ds\wedge Z^*\omega_0.
    \end{align*}
    Now for the last term on the second line we estimate
    \begin{align*}
        \int_\K \mathrm ds\wedge Z^*\omega_0=\int_{\del \K}sZ^*\omega_0= \mu \left(\int_{\T^2}Z(\mu)^*\omega_0-\int_{\T^2}Z(-\mu)^*\omega_0 \right)=0,
    \end{align*}
    since both $Z(\mu)$ and $Z(-\mu)$ are null-homotopic and $\omega_0$ is closed. Thus, $$\int_\K|\del_I Z|^2\mathrm d\mathcal V\wedge \mathrm ds = \int_\K\left(|\del_xZ|^2+|\del_yZ|^2\right)\mathrm d\mathcal V\wedge \mathrm ds.$$
    We use this to compute
    \begin{align}
        \int_\K|dZ|^2\mathrm d\mathcal V\wedge \mathrm ds &= \int_\K\left(|\del_sZ|^2+|\del_IZ|^2\right)\mathrm d\mathcal V\wedge \mathrm ds \nonumber\\
        &=\int_\K\left(|\del_sZ|^2+|\nabla^G H^\tau_s(Z)-\del_sZ|^2\right)\mathrm d\mathcal V\wedge \mathrm ds \nonumber\\
        &\leq \int_\K \left(3|\del_sZ|^2+2|\nabla^G H^\tau_s|^2\right)\mathrm d\mathcal V\wedge \mathrm ds.\label{eq:energybetween}
    \end{align}
    Note that 
    \begin{align*}
        \nabla^G H^\tau_s =\begin{pmatrix}
            A_p^{-1} & 0\\ 0 & A_p
        \end{pmatrix}\bar{\nabla}H^\tau_s,
    \end{align*}
    where $\bar{\nabla}$ denotes the gradient with respect to the Sasaki metric $\bar{g}$ defined in Section \ref{sec:hyperkahler}. The Sasaki gradient is given by 
    \[\bar{\nabla}H^\tau_s=\bar{\nabla}h+\begin{pmatrix}0\\p\end{pmatrix},\]
    so that 
    \[\nabla^G H^\tau_s = \nabla^Gh+\begin{pmatrix}0\\A_p p\end{pmatrix}.\]
    Also, since $A_p$ only has eigenvalues with modulus smaller than 1 on quotients of complex hyperbolic space (see Proposition \ref{prop:AxiCH}), we get that 
    \begin{align*}
        \left|\begin{pmatrix}0\\A_p p\end{pmatrix}\right|^2_G=g(A_p^{-1}A_pp,A_pp)\leq |p|^2_g<\delta^2.
    \end{align*}
    Filling this into Equation \eqref{eq:energybetween} gives
    \begin{align*}
        \int_\K|dZ|^2\dV\wedge ds &\leq c\int_\K\left( |\del_sZ|^2+|\nabla^G h|^2+\delta^2 \right)\dV\wedge ds\\
        &\leq C(\K)\left(E(Z)+||h||_{C^1}^2+\delta^2\right).
    \end{align*}
    Combining with \eqref{eq:Ebound} gives the desired result.
\end{proof}


\section{Bubbling analysis}\label{sec:bubbles}
Let $Z:\K\backslash S\to \mathbb D^*_\delta Q$ be the limiting map from Theorem \ref{thm:walpsuki}, $S = \Gamma\cup \mathrm{sing}(Z)$. In this section, we want to analyze the bubbling locus $\Gamma$. 

\begin{proposition}
    \label{lem:Gamma}
    Under the assumptions of Theorem \ref{thm:compactness}, we have $\Gamma =\emptyset$.
\end{proposition}
\begin{proof}
    First of all, note that $Q$ is aspherical, since its universal cover is $\CH^{2n}$. In particular, $(\mathbb D^*_\delta Q,\omega_0:= G(\cdot,I\cdot))$ is symplectically aspherical. 
    Suppose $z\in\Gamma$ is a point for which $T_z\Gamma$ exists. Then by Theorem \ref{thm:walpsuki} there exists a non-trivial $\mathcal{I}$-holomorphic sphere $\psi_z:S^2\to \mathbb D^*_\delta Q$, where $\mathcal{I}=c_0 I+c_1 J+c_2 K$ for some $c_0^2+c_1^2+c_2^2=1$. Note that $\mathcal{I}$ is compatible with $G$ so that $\omega_\mathcal{I}=G(\cdot,\mathcal{I}\cdot)$ is a symplectic form. Then
    \begin{align*}
        -\frac{1}{2}\int_{S^2}|d\psi_z|^2\text{dvol}_{S^2} &= \int_{S^2}\psi_z^*\omega_{\mathcal{I}}\\
        &= c_0\int_{S^2}\psi_z^*\omega_0+c_1\int_{S^2}\psi_z^*\omega_1+c_2\int_{S^2}\psi_z^*\omega_2.
    \end{align*}
    In the last equation the integral of $\psi_z^*\omega_0$ vanishes by asphericity. The latter two integrals vanish as well, since $\omega_1$ and $\omega_2$ are exact. Therefore, $\psi_z$ is a constant map, which yields a contradiction. This means that $\Gamma$ has no points where the tangent space exists. Since Theorem \ref{thm:walpsuki} says that the tangent space must exist for $\Haus^1$-almost every point in $\Gamma$, it follows that $\Haus^1(\Gamma)=0$.

    Recall that $\Gamma=\textup{supp}(\Theta\Haus^1\lfloor S)$. Since $\Theta$ is upper semi-continuous on a compact set, it is bounded from above. Also, $\Haus^1(S)\leq\Haus^1(\textup{sing}(Z))+\Haus^1(\Gamma)=0$. Therefore, for any open set $A$
    \[\Theta\Haus^1\lfloor S(A) = \int_{A\cap S}\Theta \mathrm d\Haus^1\leq \max\Theta \int_{S}\mathrm d\Haus^1 = 0.\]
    Thus $\Theta\Haus^1\lfloor S$ is the zero measure and its support $\Gamma$ is empty.
\end{proof}


\section{Singularities}\label{sec:sing}
Now that we know that $\Gamma$ is empty, the only thing left is to analyze the singular set $\textup{sing}(Z)$. We will need the following theorem from \cite{esfahani2023towards}. Note that in \cite{esfahani2023towards} there is no Hamiltonian term in Equation \eqref{eq:Fueter}. However, the proof applies verbatim to our setting as well.
\begin{theorem}[\cite{esfahani2023towards}]\label{thm:esfahani}
    Let $Z$ be the limiting map from Theorem \ref{thm:walpsuki} and assume $\Gamma=\emptyset$. Then for each $z\in \textup{sing}(Z)$ there exists a non-trivial contractible harmonic\footnote{The result in \cite{esfahani2023towards} actually asserts that  $\phi_z$ is tri-holomorphic, hence in particular harmonic since $G$ is hyperk\"ahler. For our purposes, we will not need the extra piece of information.} sphere $\phi_z:S^2\to \mathbb D^*_\delta Q$.
\end{theorem}
It will be useful to know how these harmonic spheres arise. Let $z\in \textup{sing}(Z)$ and $r_i$ a sequence of positive numbers converging to 0. The idea is to blow up the map $Z$ around the point $z$. Define 
\begin{align*}
    Z_{r_i}(y) = Z(\exp_z(r_i y)),
\end{align*}
which is a well-defined map on an open subset of $T_z(\R\times\T^2)$. In \cite{esfahani2023towards}, it is proven that a subsequence of these maps converge weakly to a map $\Phi_z:T_z(\R\times\T^2)\backslash\{0\}\to \mathbb D^*_\delta Q$, such that $\del_r\Phi_z=0$. Here, $\del_r$ denotes the derivative with respect to the radial direction. Now, $\phi_z$ is defined as the restriction of $\Phi_z$ to the unit sphere $S^2\subseteq T_z(\R\times\T^2)$.

Another ingredient we will need is a positive lower bound on the energy of a non-constant harmonic sphere. For \(\delta<\delta_0\), define \(\epsilon(\delta)\) to be the infimum of all \(\epsilon>0\) for which there exists a non-constant harmonic sphere of energy \(\epsilon\) whose image is contained in \(\overline{\mathbb D^*_\delta Q}\). Here the energy of a harmonic sphere \(\phi\) is
\[
\mathbb E(\phi):=\frac12\int_{S^2}|\nabla\phi|_G^2.
\]
By \cite[Theorem~3.3]{sacks1981existence}, we have \(\epsilon(\delta)\in(0,+\infty]\) for every \(\delta<\delta_0\) (with the convention \(\epsilon(\delta)=+\infty\) if no such sphere exists).

\begin{lemma}\label{lem:epsilon}
The function \(\delta\mapsto \epsilon(\delta)\) is decreasing and right-continuous.
\end{lemma}

\begin{proof}
Monotonicity is immediate: if \(\delta>\delta'\), then \(\overline{\mathbb D^*_{\delta'}Q}\subset \overline{\mathbb D^*_\delta Q}\), hence \(\epsilon(\delta)\leq \epsilon(\delta')\).
It remains to show right-continuity. Since \(\epsilon\) is decreasing, it suffices to prove right lower semicontinuity. Arguing by contradiction, assume that \(\epsilon\) fails to be right lower semicontinuous at some \(\delta_*<\delta_0\). Then there exist \(\alpha>0\) and a sequence \(\delta_i\downarrow \delta_*\) such that
\[
\epsilon(\delta_i)\leq \epsilon(\delta_*)-\alpha \qquad\text{for all }i.
\]
By definition of \(\epsilon(\delta_i)\), for each \(i\) there exists a non-constant harmonic sphere \(\phi_i:S^2\to \mathbb D^*_{\delta_0}Q\) with \(\mathrm{im}(\phi_i)\subset \overline{\mathbb D^*_{\delta_i}Q}\) and
\[
\mathbb E(\phi_i)\leq \epsilon(\delta_*)-\alpha.
\]
Since \(\delta_i\leq \delta_1\) for all \(i\), we also have \(\mathrm{im}(\phi_i)\subset \overline{\mathbb D^*_{\delta_1}Q}\), and hence \(\mathbb E(\phi_i)\geq \epsilon(\delta_1)>0\) (the inequality is meaningful because the existence of the maps \(\phi_i\) forces \(\epsilon(\delta_1)<+\infty\)).
By bubble convergence for harmonic maps \cite{parker1996bubble}, after passing to a subsequence there exists a non-constant harmonic sphere \(\phi_*:S^2\to \mathbb D^*_{\delta_0}Q\) such that \(\mathrm{im}(\phi_*)\subset \overline{\mathbb D^*_{\delta_*}Q}\) and
\[
0<\epsilon(\delta_1)\leq \mathbb E(\phi_*)\leq \liminf_{i\to\infty}\mathbb E(\phi_i)\leq \epsilon(\delta_*)-\alpha < \epsilon(\delta_*).
\]
This contradicts the definition of \(\epsilon(\delta_*)\). Therefore \(\epsilon\) is right lower semicontinuous, hence right-continuous.
\end{proof}

As a consequence of Lemma~\ref{lem:epsilon}, we have
\[
\epsilon(0):=\lim_{\delta\downarrow 0}\epsilon(\delta)\in(0,+\infty].
\]
In the setting of Theorem~\ref{thm:compactness} we have \(\epsilon(0)=+\infty\): indeed, by assumption \(Q\) admits no non-constant harmonic spheres, and any harmonic sphere for \(G\) whose image lies in the zero section is harmonic for the K\"ahler metric \(g\) on \(Q\).
We emphasize, however, that even if \(Q\) admits no non-constant harmonic spheres, this does \emph{not} automatically exclude non-constant harmonic spheres for the Feix metric \(G\) inside \(\mathbb D^*_{\delta_0}Q\), unless one is in the flat case \(Q=\T^{2n}\). The reason is that \(G\) is Ricci-flat (as any hyperk\"ahler metric), so its sectional curvature has no definite sign. The point of the next lemma is that, provided \(h\) is sufficiently small and the image of \(Z\) is contained in a sufficiently small \(\mathbb D^*_\delta Q\), the singular set cannot detect non-constant harmonic spheres: any harmonic sphere arising from the blow-up analysis would have energy so small that it falls below \(\epsilon(\delta)\), which is impossible by definition.

For \(s\in\mathbb R\), let
\[
\K_s=[s-\mu_0,s+\mu_0]\times \T^2,
\]
where \(\mu_0=r_0+1\) and \(r_0=\pi\) is the injectivity radius of \(\R\times\T^2\). Let \(C_1=C(\K_s)\) be the constant from Lemma~\ref{lem:energybound}; note that \(C_1\) is independent of \(s\). In what follows,  \(\alpha>0\) is chosen sufficiently small so that Equations \eqref{eq:delta1} and \eqref{eq:delta2} below  are satisfied for some $\delta>0$. 

\begin{proposition}\label{lem:sing}
Under the assumptions of Theorem \ref{thm:compactness}, pick \(\delta>0\) such that:
\begin{align}
\delta &< \min \Big \{ \delta_0, \frac{\alpha}{2C_1}\Big \}, \label{eq:delta1}\\
\epsilon(\delta) &> \alpha. \label{eq:delta2}
\end{align}
Assume further that:
\begin{enumerate}[label=\alph*)]
\item all elements in $\mathcal M$ have image contained in  \(\mathbb D^*_\delta Q\), and
\item \(\|h\|_{C^1}^2+ \|h\|_{C^1}<\dfrac{\alpha}{2C_1}=:C_*.\)
\end{enumerate}
Let \(Z\) be a limiting map as in Theorem \ref{thm:walpsuki}. Then \(\mathrm{sing}(Z)=\emptyset\).
\end{proposition}

\begin{proof}
Assume by contradiction that \(z=(s,x,y)\in \mathrm{sing}(Z)\). By Theorem \ref{thm:esfahani} there exists a non-constant harmonic sphere \(\phi_z:S^2\to \mathbb D^*_\delta Q\) detected at \(z\). Moreover, as shown in \cite{esfahani2023towards}, one has
\[
\frac12\int_{B_1(0)}|\nabla Z_{r_i}|^2 \leq \frac12\int_{B_{r_0}(z)}|\nabla Z|^2.
\]
The right-hand side is bounded by \(\E_{\K_s}(Z)\). Passing to the blow-up limit \(\Phi_z\), this yields
\[
\frac12\int_{B_1(0)}|\nabla\Phi_z|^2 \le \E_{\K_s}(Z).
\]

Since \(\partial_r\Phi_z=0\), the map \(\Phi_z\) is constant along the radial direction. Denote by \(\phi_z^\rho\) the restriction of \(\Phi_z\) to the sphere \(\rho S^2\) of radius \(\rho\). Then the energy of \(\Phi_z\) on \(B_1(0)\) can be written as
\[
\int_{B_1(0)}|\nabla\Phi_z|^2
=\int_{\rho=0}^1\int_{\rho S^2}|\nabla \phi^\rho_z|^2
=\int_{S^2}|\nabla \phi_z|^2,
\]
so that
\[
\mathbb E(\phi_z)=\frac12\int_{S^2}|\nabla\phi_z|^2
=\frac12\int_{B_1(0)}|\nabla\Phi_z|^2
\le \E_{\K_s}(Z).
\]
Finally, Lemma~\ref{lem:energybound} and the assumptions imply
\[
\mathbb E(\phi_z)
\le C_1\bigl(\delta^2+\|h\|_{C^1}+\|h\|_{C^1}^2\bigr)
< \alpha
< \epsilon(\delta),
\]
where the last inequality uses \eqref{eq:delta2}. This contradicts the definition of \(\epsilon(\delta)\), since \(\phi_z\) is non-constant with image contained in \(\overline{\mathbb D^*_\delta Q}\). Hence \(\mathrm{sing}(Z)=\emptyset\).
\end{proof}

\begin{remark}
Proposition \ref{lem:sing} remains valid, with only minor modifications, in the case $\epsilon(0) < +\infty$, for instance when $Q$ is a compact symmetric space. In that situation one simply chooses $\alpha < \epsilon(0)$ and proceeds as above. Clearly, additional conditions (e.g. smallness of $h$) are needed in order to have $\Gamma =\emptyset$. We refrain to pursue this direction here, and leave it for future research. 
\end{remark}

\begin{proof}[Proof of Theorem \ref{thm:compactness}]
    Pick $0<\delta_*<\delta_0$ such that Equations \eqref{eq:delta1} and \eqref{eq:delta2} are satisfied for all $\delta<\delta_*$ and choose $C_*$ as in Proposition \ref{lem:sing}. Take a sequence $\{(\tau_m,Z^m)\}_{m\in \N}$ in $\M$. We want to prove convergence of a subsequence in $C^\infty$ on all $\K=[-\mu,\mu]\times\T^2$. By Lemma \ref{lem:energybound} there is a bound on $\E_\K(Z^m)$, uniform in $m$. Thus, according to Theorem \ref{thm:walpsuki}, there is a limiting map $Z$ defined on $\K\backslash S$. The fact that $S=\Gamma\cup \text{sing}(Z)=\emptyset$ follows from Propositions \ref{lem:Gamma} and \ref{lem:sing}.
\end{proof}

\begin{remark}
A better understanding of the function $\delta \mapsto \epsilon(\delta)$ for the Biquard--Gauduchon metric $G$ on compact quotients of complex hyperbolic space could lead to a substantial improvement of Lemma~\ref{lem:sing}, and hence of Theorem~\ref{thm:compactness}. 
In the present argument, the smallness of $\delta$ is required in order to guarantee that the energy of any harmonic sphere detected by the singular set is strictly smaller than $\epsilon(\delta)$. If one could obtain sharper lower bounds on $\epsilon(\delta)$, the restriction on $\delta$ could be relaxed. For instance, if one were able to show that 
\[
\epsilon(\delta) \ge C_1 \delta_0
\quad \text{for all } \delta \le \delta_0/2,
\]
then one could choose $\alpha = C_1 \delta_0$ and simply impose $\delta_* = \delta_0/2$. This would remove the need for quantitative $L^\infty$-estimates in the proof of compactness.
We stress, however, that since the metric $G$ degenerates as $|p|_g \to \delta_0$, it is conceivable that $\epsilon(\delta) \to 0$ as $\delta \uparrow \delta_0$. A precise analysis of the behaviour of $\epsilon(\delta)$ for the Biquard--Gauduchon metric appears to be an interesting geometric problem in its own right and will be the subject of future investigation.
\end{remark}


\section{A note on the a priori $L^\infty$-estimates}
\label{sec:bounds}

The next step towards extending Theorem~\ref{thm:Arnold} to non-flat targets is to establish the a priori \(L^\infty\)-bounds required to apply Theorem~\ref{thm:compactness}. As already explained after the statement of Theorem~\ref{thm:compactness}, what we actually need -- at least at present -- is a \emph{quantitative} \(L^\infty\)-estimate: namely, we want to ensure that choosing the perturbation \(h\) sufficiently small (in an appropriate norm) forces every Fueter/Floer trajectory in \(\mathcal M\) to remain confined to a sufficiently small neighborhood of the zero section in \(T^*Q\). 
We begin this chapter by proving such a confinement result---to the best of our knowledge, not previously available in the literature---in the model case \(Q=\T^{2n}\) of the flat torus.

\begin{theorem}\label{thm:C0boundflat}
Let \(Q=\T^{2n}\) be a flat torus, and let \((G,I,J,K)\) be the standard hyperk\"ahler structure on \(T^*\T^{2n}\).
Let \(H_s:\R\times \T^2\times T^*\T^{2n}\to\R\) be defined by
\[
H_s(t,q,p)=\frac{1}{2}|p|^2+\beta(s)\,h(t,q,p),
\]
where \(\|h\|_{C^2}<+\infty\). Then, for every \(\delta\in(0,+\infty)\) there exists
\(\nu=\nu\bigl(\delta,\|h\|_{C^2}\bigr)>0\) such that, if \(\|h\|_{C^1}<\nu\), every solution
\(Z:\R\times \T^2\to T^*\T^{2n}\) of
\begin{equation}\label{eq:FueterFloer}
\partial_s Z+J\partial_x Z+K\partial_y Z=\nabla^G H_s(Z)
\end{equation}
satisfying \(|Z(s,\cdot)|\to 0\) as \(s\to\pm\infty\) has image contained in \(\mathbb D^*_{\delta}\T^{2n}\).
\end{theorem}

\begin{proof}
We divide the proof of Theorem~\ref{thm:C0boundflat} into four steps.  
Let \(Z:\R\times \T^2 \to T^*\T^{2n}\) be as in the statement and set
\begin{equation}\label{eq:Phi}
\Phi:\R\times \T^2 \to \R, \qquad \Phi(s,x,y):=\frac12\,|Z(s,x,y)|^2.
\end{equation}
Throughout the proof, \(C>0\) denotes a constant whose value may change from line to line.

\medskip

\noindent\textbf{Step 1 (\(\Phi\) is a subsolution of an elliptic PDE).}
There exists a constant \(\alpha=\alpha(\|h\|_{C^2})>0\) such that \(\Phi\) satisfies
\begin{equation}\label{eq:subsolPhi}
    -\Delta \Phi + \del_s \Phi - \alpha \Phi \leq \alpha,
\end{equation}
where \(\Delta:=\del_s^2+\del_x^2+\del_y^2\).

Following \cite{CRPS}, in the coordinates \((q,p)\) on \(T^*\T^{2n}\) the Fueter/Floer equation
\eqref{eq:FueterFloer} can be written as
\begin{equation}\label{eq:Hameqtorus}
\left\{
\begin{array}{l}
\del_s q = \del_q h + 2\bar\del\, p,\\[2pt]
\del_s p = p + \del_p h - 2\del\, q,
\end{array}
\right.
\end{equation}
where \(\del\) and \(\bar\del\) are the usual Cauchy--Riemann and anti-Cauchy--Riemann operators,
and for notational simplicity we write \(h(s,\cdot)\) in place of \(\beta(s)\,h(\cdot)\).
A direct computation using \eqref{eq:Hameqtorus} yields
\begin{equation}\label{eq:Deltap}
\Delta p = \del_s p + \del_s(\del_p h) - 2\del(\del_q h).
\end{equation}
Using \(\Phi=\frac12|p|^2\) (since \(|Z|^2=|p|^2\) in the flat model) we obtain
\begin{equation}\label{eq:firstpde}
\Delta \Phi
= |\nabla p|^2 + \del_s\Phi
+ \underbrace{\langle p,\del_s(\del_p h)\rangle}_{=:(*)}
-\underbrace{2\langle p,\del(\del_q h)\rangle}_{=:(**)}.
\end{equation}
We estimate \((*)\). Using \eqref{eq:Hameqtorus}, the inequality
\(ab\le \frac12(\nu a^2+\nu^{-1}b^2)\) with \(\nu=2\|h\|_{C^2}\), and
\(|\bar\del p|^2\le \frac12(|\del_x p|^2+|\del_y p|^2)\), we get after some work
\begin{align*}
|(*)|
&\le |\langle p, \del^2_{sp} h\rangle|
   + |\langle p, \del^2_{qp} h \cdot \del_s q\rangle|
   + |\langle p, \del^2_{pp} h \cdot \del_s p\rangle| \\
&\le \|h\|_{C^2}\Big(|p| + |p|\,|\del_s q| + |p|\,|\del_s p|\Big) \\
&\le \|h\|_{C^2}\Big(|p| + |p|\,|\del_q h| + 2|p|\,|\bar\del p| + |p|\,|\del_s p|\Big) \\
&\leq \big(\|h\|_{C^2}+\|h\|_{C^2}^2\big)\sqrt{2\Phi}
   + 4\|h\|_{C^2}^2\Phi
   + \frac12\big(|\del_x p|^2+|\del_y p|^2\big)
   + \frac14|\del_s p|^2.
\end{align*}
A similar computation gives the bound
\begin{align*}
|(**)|
&\le \big(2\|h\|_{C^2}+\|h\|_{C^2}^2\big)\sqrt{2\Phi}
 + 2\big(\|h\|_{C^2}+2\|h\|_{C^2}^2\big)\Phi
 + \frac12\big(|\del_s p|^2+|\del_x p|^2+|\del_y p|^2\big).
\end{align*}
Inserting these estimates into \eqref{eq:firstpde} and using the elementary inequality
\(\sqrt{\Phi}\le 1+\Phi\), we obtain after rearranging
\[
\Delta \Phi \ge \del_s\Phi - 16\|h\|_{C^2}\bigl(1+\|h\|_{C^2}\bigr)(1+\Phi).
\]
Thus \eqref{eq:subsolPhi} holds with
\[
\alpha(\|h\|_{C^2}) := 16\|h\|_{C^2}\bigl(1+\|h\|_{C^2}\bigr).
\]

\noindent\textbf{Step 2 (A mean value inequality for \(\Phi\)).}
There exist \(R_0=R_0(\alpha)\le 1\) and \(C_{\mathrm{mv}}>0\) such that for every
\(\xi_0\in \R\times\T^2\) and every \(r\in(0,R_0]\),
\begin{equation}\label{eq:meanvaluePhi}
\|\Phi\|_{L^\infty(B_r(\xi_0))}
\le C_{\mathrm{mv}}\Big(\fint_{B_{2r}(\xi_0)} \Phi + \alpha r^2\Big).
\end{equation}

Fix \(\xi_0\in\R\times\T^2\) and \(r\in(0,R_0]\), with \(R_0\le 1\) to be chosen below.
Consider the rescaled function
\[
\psi:B_2(0)\subset\R^3\to\R,\qquad \psi(z):=\Phi(\xi_0+rz),
\]
where we identify \(B_2(\xi_0)\subset \R\times\T^2\) with a ball in \(\R^3\) since \(r\le 1\).
Then \(\psi\) satisfies
\[
-\Delta\psi + r\,\del_s\psi - \alpha r^2\psi \le \alpha r^2.
\]
Writing \(z=(\sigma,x,y)\) and setting
\[
\eta:B_2(0)\to\R,\qquad \eta(z):=e^{-\frac r2 \sigma}\psi(z),
\]
we note that \(|\eta|\le e|\psi|\le e^2|\eta|\) on \(B_2(0)\). Hence it suffices to prove
\begin{equation}\label{eq:infty1eta}
\|\eta\|_{L^\infty(B_1(0))}
\le C\Big(\fint_{B_2(0)} \eta + \alpha r^2\Big).
\end{equation}
A short computation shows that \(\eta\) satisfies
\[
(-\Delta+\mu)\eta \le \alpha r^2 e^{-\frac r2\sigma},
\qquad\text{where}\qquad
\mu:=\frac{r^2}{4}-\alpha r^2.
\]
Let \(\tau:B_2(0)\to\R\) be the (unique) solution\footnote{Existence and uniqueness for this Dirichlet problem are standard; see e.g.\ \cite{E}.} of
\[
\left\{
\begin{array}{rcl}
(-\Delta+\mu)\tau &=& e^{-\frac r2\sigma}\qquad\text{in }B_2(0),\\
\tau&=&0\qquad\qquad\ \ \text{on }\partial B_2(0),
\end{array}
\right.
\]
and set \(W:=\eta-\alpha r^2\tau\). Then \( (-\Delta+\mu)W\le 0\), hence also
\((-\Delta+\mu)W^+\le 0\) in the sense of distributions, where \(W^+:=\max\{W,0\}\).
Since \(0\le \eta = W + \alpha r^2\tau \le W^+ + \alpha r^2\tau\), estimate \eqref{eq:infty1eta}
follows once we show:
\begin{align*}
\textit{i)}\quad & 0\le \tau \le C \ \text{ on }B_2(0),\\
\textit{ii)}\quad & \|W^+\|_{L^\infty(B_1(0))} \le C \fint_{B_2(0)} W^+.
\end{align*}
Indeed, \textit{i)} implies \(W^+\le \eta\) and \(\eta\le W^+ + C\alpha r^2\), while \textit{ii)} gives
\[
\|\eta\|_{L^\infty(B_1(0))}
\le \|W^+\|_{L^\infty(B_1(0))} + C\alpha r^2
\le C\Big(\fint_{B_2(0)}\eta + \alpha r^2\Big).
\]
We prove \textit{i)} and \textit{ii)} for \(\alpha>\frac14\); the case \(\alpha\le \frac14\) is simpler since then \(\mu\ge 0\), and both statements follow by direct comparison and the mean value property for subharmonic functions.
Set
\[
R_0=R_0(\alpha):=\min\Big\{1,\frac{\pi}{\sqrt{2(4\alpha-1)}}\Big\},
\]
so that for all \(r\in(0,R_0]\),
\[
0\le -\mu=\alpha r^2-\frac{r^2}{4}\le \alpha R_0^2-\frac{R_0^2}{4}\le \frac{\pi^2}{8}=\frac12\lambda_1,
\]
where \(\lambda_1:=\frac{\pi^2}{4}\) is the first Dirichlet eigenvalue of \(-\Delta\) on \(B_2(0)\).
Testing \((-\Delta+\mu)\tau\ge 0\) against \(\tau^-:=\max\{-\tau,0\}\in H^1_0(B_2(0))\) yields
\[
\int_{B_2(0)}|\nabla\tau^-|^2 + \mu\int_{B_2(0)}|\tau^-|^2 \le 0.
\]
By Poincar\'e inequality, \(\lambda_1\int|\tau^-|^2\le \int|\nabla\tau^-|^2\), hence
\[
(\lambda_1+\mu)\int_{B_2(0)}|\tau^-|^2\le 0.
\]
Since \(\lambda_1+\mu\ge \frac12\lambda_1>0\), we conclude \(\tau^-\equiv 0\), i.e.\ \(\tau\ge 0\).
Next, let \(\kappa\) be the (unique) solution\footnote{See again \cite{E}.} of
\[
\left\{
\begin{array}{rcl}
(-\Delta+\mu)\kappa &=& 1\qquad\text{in }B_2(0),\\
\kappa&=&0\qquad\text{on }\partial B_2(0).
\end{array}
\right.
\]
A direct computation shows \(\kappa(z)=f(\rho)\) with \(\rho=|z|\) and
\[
f(\rho) = -\frac 1\mu\Big(\frac{2 \sin (\sqrt{-\mu}\,\rho)}{\rho \sin (2\sqrt{-\mu})} - 1\Big).
\]
In particular, \(0\le \kappa\le \kappa(0)\) on \(B_2(0)\), and using that \(\kappa(0)\) is decreasing in \(\mu\in\big[-\frac{\pi^2}{8},0\big)\) we obtain the uniform bound
\[
0\le \kappa(z)\le \kappa(0)
\le \frac{4\sqrt 2}{\pi \sin (\frac \pi{\sqrt 2}) } - \frac{8}{\pi^2} \sim 1.45.
\]
Since \((-\Delta+\mu)\tau = e^{-\frac r2\sigma}\le e\) on \(B_2(0)\), we have
\[
(-\Delta+\mu)(e\kappa-\tau) \ge e- e = 0,
\qquad\text{and}\qquad
(e\kappa-\tau)|_{\partial B_2(0)}=0,
\]
so the previous argument implies \(e\kappa-\tau\ge 0\). Hence
\[
0\le \tau(z)\le e\,\kappa(z)\le e\,\|\kappa\|_{L^\infty(B_2(0))}\sim 1.45\,e,
\]
which proves \textit{i)}.
For \textit{ii)}, let \(\varphi_1\) be the first Dirichlet eigenfunction of \(-\Delta\) on \(B_2(0)\),
namely the radial function
\[
\varphi_1(\rho)=\frac{\sin(\frac{\pi}{2}\rho)}{\rho}.
\]
It is strictly decreasing on \([0,2]\), and in particular on \([0,\frac32]\) we have
\begin{equation}\label{eq:boundphi1}
\frac{\pi}{2}=\varphi_1(0)\ \ge\ \varphi_1(\rho)\ \ge\ \varphi_1\Big(\frac32\Big)=\frac{\sqrt2}{3}.
\end{equation}
Since \((-\Delta+\mu)W^+\le 0\) on \(B_2(0)\), this holds in particular on \(B_{3/2}(0)\).
Define
\[
V:B_{3/2}(0)\to\R,\qquad V:=\frac{W^+}{\varphi_1}.
\]
By \eqref{eq:boundphi1}, it suffices to show
\begin{equation}\label{eq:meanvalueV}
\|V\|_{L^\infty(B_1(0))} \le C \fint_{B_{3/2}(0)} V.
\end{equation}

\begin{remark}
We work on \(B_{3/2}(0)\subset B_2(0)\) since \(\varphi_1\) vanishes on \(\partial B_2(0)\);
consequently, the operator introduced below is not uniformly elliptic up to \(\partial B_2(0)\).
Once \eqref{eq:meanvalueV} is established, it implies
$$\|W^+\|_{L^\infty(B_1(0))}\le C\fint_{B_{3/2}(0)}W^+,$$
and the latter yields \textit{ii)} after a trivial comparison between averages.
\end{remark}

Formally computing on \(B_{3/2}(0)\),
\begin{align*}
(-\Delta+\mu)W^+
&=(-\Delta+\mu)(\varphi_1 V) \\
&= -\varphi_1\Delta V - 2\nabla\varphi_1\cdot\nabla V - V\Delta\varphi_1 + \mu\varphi_1 V \\
&= -\varphi_1\Delta V - 2\nabla\varphi_1\cdot\nabla V + (\lambda_1+\mu)\varphi_1 V,
\end{align*}
hence multiplying by \(-\varphi_1\) we obtain
\[
\varphi_1^2\Delta V + 2\varphi_1\nabla\varphi_1\cdot\nabla V - (\lambda_1+\mu)\varphi_1^2 V
= -\varphi_1(-\Delta+\mu)W^+ \ge 0.
\]
Therefore \(V\) is a subsolution of an operator \(L\) of the form (9.1) in \cite{GT}, with zeroth-order term
\(c=-(\lambda_1+\mu)\varphi_1^2\le 0\). Applying Theorem~9.20 in \cite{GT} (with the standard extension to
\(B_R\subset B_{\theta R}\), here \(\theta=\frac32\)) yields \eqref{eq:meanvalueV}, hence \textit{ii)}.
This completes Step~2.

\noindent\textbf{Step 3 (An \(L^\infty\)-bound for \(\Phi\) in terms of \(\alpha\) and \(\|h\|_{C^1}\)).}
There exists a constant \(C_\infty>0\) such that
\begin{equation}\label{eq:C0bound}
\|\Phi\|_{L^\infty} \le C_\infty\,\alpha^{\frac35}\,\|h\|_{C^1}^{\frac25}.
\end{equation}
We may assume without loss of generality that
\[
\|h\|_{C^1}\le \min\Big\{1,\ \frac{2\alpha}{3}R_0^5\Big\},
\]
where \(R_0=R_0(\alpha)\le 1\) is as in Step~2.
To derive \eqref{eq:C0bound}, it suffices to prove that there exists \(C>0\) such that
\begin{equation}\label{eq:boundhC1}
\int_{B_{2r}(\xi_0)} \Phi \le C\,\|h\|_{C^1}
\qquad\text{for all }\xi_0\in\R\times\T^2,\ r\in(0,R_0].
\end{equation}
Indeed, combining \eqref{eq:boundhC1} with \eqref{eq:meanvaluePhi} gives
\[
\|\Phi\|_{L^\infty(B_r(\xi_0))}
\le C\Big(\frac{\|h\|_{C^1}}{r^3}+\alpha r^2\Big).
\]
Optimizing in \(r\) yields \(r_*:=\big(\frac{3\|h\|_{C^1}}{2\alpha}\big)^{1/5}\le R_0\) and
\[
\|\Phi\|_{L^\infty(B_{r_*}(\xi_0))}
\le C\Big(\big(\tfrac23\big)^{\frac35}+\big(\tfrac32\big)^{\frac25}\Big)\alpha^{\frac35}\|h\|_{C^1}^{\frac25}.
\]
Covering \(\R\times\T^2\) by balls of radius \(r_*\) then gives \eqref{eq:C0bound}.
It remains to prove \eqref{eq:boundhC1}. For each \(s\in\R\), define the average
\[
\bar p(s):=\int_{\T^2} p(s,x,y)\,dx\,dy.
\]
Integrating the second equation in \eqref{eq:Hameqtorus} over \(\T^2\) we obtain
\[
\bar p'(s)=\bar p(s) + \underbrace{\int_{\T^2}\del_p h\bigl(s,\cdot,q(s,\cdot),p(s,\cdot)\bigr)}_{=:a(s)}
-2\underbrace{\int_{\T^2}\del q}_{=0}.
\]
Clearly \(|a(s)|\le \|h\|_{C^1}\) for all \(s\). The asymptotic condition \(|Z(s,\cdot)|\to 0\) as \(s\to\pm\infty\)
implies \(|\bar p(s)|\to 0\) as \(s\to\pm\infty\), and hence
\begin{equation}
\bar p(s)=-\int_s^{+\infty} e^{s-t}a(t)\,dt,
\qquad\text{so that}\qquad
|\bar p(s)|\le \|h\|_{C^1}\quad\forall s\in\R.
\label{eq:boundbarp}
\end{equation}
Now set \(p_0(s,\cdot):=p(s,\cdot)-\bar p(s)\). By Poincar\'e inequality on \(\T^2\),
\[
\|p_0(s,\cdot)\|_{L^2(\T^2)}^2 \le 4\|\bar\del p_0(s,\cdot)\|_{L^2(\T^2)}^2
=4\|\bar\del p(s,\cdot)\|_{L^2(\T^2)}^2.
\]
Integrating over \([s_0-2r,s_0+2r]\) and using again \eqref{eq:Hameqtorus}, we find
\begin{align}
\int_{s_0-2r}^{s_0+2r} & \|p_0(s,\cdot)\|_{L^2(\T^2)}^2\,ds \nonumber \\
&\le 4\int_{s_0-2r}^{s_0+2r}\|\bar\del p(s,\cdot)\|_{L^2(\T^2)}^2\,ds \nonumber \\
&\le 2\int_{s_0-2r}^{s_0+2r}\|\del_q h(s,\cdot)\|_{L^2(\T^2)}^2\,ds
    +2\int_{s_0-2r}^{s_0+2r}\|\del_s q(s,\cdot)\|_{L^2(\T^2)}^2\,ds \label{eq:boundp0}\\
&\le 8r\,\|h\|_{C^1}^2 + 2\int_{\R\times\T^2}|\del_s Z|^2. \nonumber
\end{align}
Combining \eqref{eq:boundbarp} and \eqref{eq:boundp0} gives
\begin{align*}
\int_{B_{2r}(\xi_0)}\Phi
&\le \frac12\int_{s_0-2r}^{s_0+2r}\|p(s,\cdot)\|_{L^2(\T^2)}^2\,ds \\
&\le \int_{s_0-2r}^{s_0+2r}|\bar p(s)|^2\,ds
    +\int_{s_0-2r}^{s_0+2r}\|p_0(s,\cdot)\|_{L^2(\T^2)}^2\,ds \\
&\le 4r\,\|h\|_{C^1}^2 + 8r\,\|h\|_{C^1}^2 + 2\int_{\R\times\T^2}|\del_s Z|^2 \\
&= 12r\,\|h\|_{C^1}^2 + 4\|h\|_{\mathrm{Hofer}} \\
&\le 12r\,\|h\|_{C^1}^2 + 8\|h\|_{C^1}
\le 20\,\|h\|_{C^1},
\end{align*}
where we used: for Hamiltonians of the form \(h(s,\cdot)=\beta(s)h(\cdot)\) one has
\(\int_{\R\times\T^2}|\del_s Z|^2\le 2\|h\|_{\mathrm{Hofer}}\) (see e.g.\ \cite{albers2016cuplength}),
\(\|h\|_{\mathrm{Hofer}}\le 2\|h\|_{C^1}\), and \(r\le R_0\le 1\) together with \(\|h\|_{C^1}\le 1\).
This proves \eqref{eq:boundhC1}, hence completes Step~3.

\medskip

\noindent\textbf{Step 4 (Conclusion).}
By Step~3, it suffices to impose
\[
\|h\|_{C^1} < \min\Big\{1,\ \frac{2\alpha}{3}R_0^5,\ \Big(\frac{\delta^2}{2C_\infty}\Big)^{\frac52}\alpha^{-\frac32}\Big\}
=: \nu(\delta,\|h\|_{C^2}),
\]
which ensures \(\|\Phi\|_{L^\infty}\le \frac{\delta^2}{2}\), i.e.\ \(|Z|\le \delta\) everywhere. \qedhere
\end{proof}

\subsection*{Concluding remarks in the non-flat case.} We now turn our attention to the case where $(Q,g,i)$ is a compact quotient of complex hyperbolic space. Let $G$ denote the Biquard--Gauduchon hyperk\"ahler metric on $\mathbb D^*_{\delta_0}Q$, and let $(I,J,K)$ be the corresponding complex structures, where $I$ is naturally induced by $i$. 
As above, let $H_s$ be defined by 
\[
H_s(t,q,p) = \frac 12 |p|_g^2 + \beta(s) h(t,q,p),
\]
where $h$ has finite $C^2$-norm with respect to $G$ (for instance, $h$ may have compact support in $\mathbb D^*_{\delta_0}Q$). We consider the Fueter/Floer equation 
\begin{equation}
    \del_s Z + J(Z) \del_x Z + K(Z) \del_y Z = \nabla^G H_s(Z).
\label{eq:FueterFloerBG}
\end{equation}
What we would need are statements of the following type: 
\begin{enumerate}
    \item \textbf{Qualitative a priori estimates:} There exists a function $\delta_1:[0,+\infty) \to [0,\delta_0)$ such that 
    \[
    \mathrm{im}(Z)\subset \mathbb D^*_{\delta_1(\|h\|_{C^2})}Q
    \]
    for every solution $Z$ of \eqref{eq:FueterFloerBG} satisfying $|Z(s,\cdot)|_g \to 0$ as $s\to \pm\infty$. 
    
    \item \textbf{Quantitative a priori estimates:} For every $\delta \in (0,\delta_1)$ there exists $\nu=\nu(\delta,\delta_1)>0$ such that, if $\|h\|_{C^1}<\nu$, then every solution 
    \[
    Z:\R\times \T^2 \to \mathbb D^*_{\delta_1} Q
    \]
    of \eqref{eq:FueterFloerBG} with $|Z(s,\cdot)|_g \to 0$ as $s\to \pm\infty$ has image contained in $\mathbb D^*_\delta Q$.
\end{enumerate}

Statement 1 is needed in order to work with bounded geometric quantities when proving Statement 2, since the Biquard-Gauduchon metric explodes as one approaches the boundary of $\mathbb D^*_{\delta_0} Q$.
As in Theorem~\ref{thm:C0boundflat}, we consider the function 
$\Phi: \R\times \T^2 \to \R$
defined in \eqref{eq:Phi}, which we write as the composition of $Z$ with 
$H^0:T^*Q\to \R,\ H^0(q,p)= \frac 12 |p|_g^2.$
Following \cite[Lemma 9.2.2]{jost2005riemannian}, we obtain 
\begin{equation}
\Delta (H^0\circ Z) 
= \mathrm{tr} \big (\nabla^2H^0[\mathrm d Z,\mathrm d Z]\big) 
+ G(\nabla^G H^0(Z), \tau(Z) ),
\label{eq:Jost}
\end{equation}
where $\nabla^2H^0$ is the Riemannian Hessian of $H^0$ and $\tau(Z)$ is the tension field of $Z$.
We first analyze the second term on the right-hand side. As already shown in the proof of Lemma \ref{lem:energybound}, the gradient $\nabla^G H^0(Z)$ is purely vertical, and its vertical component is given by $A_Z Z$. The following lemma provides an explicit formula for the tension field $\tau(Z)$ of $Z$. 

\begin{lemma}\label{lem:tension}
Let $Z:\R\times \T^2\to \mathbb D^*_{\delta_0}Q$ be a solution of \eqref{eq:FueterFloerBG}, and denote $Z_s=\del_s Z$, $Z_x=\del_x Z$, $Z_y=\del_y Z$. Let $\nabla$ be the Levi--Civita connection of $G$. Then, the tension field $\tau(Z)$ of $Z$ is given by
\begin{align}
\tau(Z) 
&= 2\nabla^2H_s(Z)[Z_s] - \nabla^2 H_s(Z)[\nabla^GH(Z)] \nonumber \\ & \ \ \ + \nabla^G(\del_s H_s)(Z) - J(Z) \nabla^G(\del_x H_s)(Z) - K(Z) \nabla^G (\del_yH_s)(Z),
\label{eq:tensionfield1}
\end{align}
and hence in particular 
$$|\tau(Z)| \lesssim \|H_s\|_{C^2}\big (|Z_s|+ \|H_s\|_{C^1}\big).$$
\end{lemma}

\begin{proof}
Since $\nabla$ is torsion-free and $[\del_s,\del_x]=[\del_s,\del_y]=[\del_x,\del_y]=0$, we have the commutation relations for the pullback connection:
\begin{equation}\label{eq:comm}
\nabla_s Z_x = \nabla_x Z_s,\qquad
\nabla_s Z_y = \nabla_y Z_s,\qquad
\nabla_x Z_y = \nabla_y Z_x,
\end{equation}
where we have set $\nabla_s =\nabla_{Z_s}, \nabla_x=\nabla_{Z_x},\nabla_y=\nabla_{Z_y}.$
With respect to the flat metric on $\R\times\T^2$, the tension field is
\[
\tau(Z)=\nabla_s Z_s+\nabla_x Z_x+\nabla_y Z_y.
\]
Using the Fueter/Floer equation \eqref{eq:FueterFloerBG} we write
\[
Z_s=-J(Z)Z_x-K(Z)Z_y+\nabla^G H_s(Z).
\]
Differentiating in the $s$-direction and using $\nabla J=\nabla K=0$, we obtain
\begin{align*}
\nabla_s Z_s
&= -J(Z)\nabla_s Z_x - K(Z)\nabla_s Z_y + \nabla_s\big(\nabla^G H_s(Z)\big) \\
&= -J(Z)\nabla_x Z_s - K(Z)\nabla_y Z_s + \nabla_s\big(\nabla^G H_s(Z)\big),
\end{align*}
where we used \eqref{eq:comm}. Substituting again and expanding yields
\begin{align*}
\nabla_s Z_s
&= -J(Z)\nabla_x\big(-J(Z)Z_x-K(Z)Z_y+\nabla^G H_s(Z)\big) \\
&\quad -K(Z)\nabla_y\big(-J(Z)Z_x-K(Z)Z_y+\nabla^G H_s(Z)\big)
+ \nabla_s\big(\nabla^G H_s(Z)\big) \\
&= -\nabla_x Z_x - \nabla_y Z_y
- J(Z)\nabla_x\big(\nabla^G H_s(Z)\big)
- K(Z)\nabla_y\big(\nabla^G H_s(Z)\big) \\
&\quad -JK\,\nabla_x Z_y - KJ\,\nabla_y Z_x
+ \nabla_s\big(\nabla^G H_s(Z)\big).
\end{align*}
The mixed terms cancel because $\nabla_x Z_y=\nabla_y Z_x$ by \eqref{eq:comm} and $JK=-KJ$. Hence
\begin{align*}
\nabla_s Z_s
&= -\nabla_x Z_x - \nabla_y Z_y \\
&\quad - J(Z)\nabla_x\big(\nabla^G H_s(Z)\big)
- K(Z)\nabla_y\big(\nabla^G H_s(Z)\big)
+ \nabla_s\big(\nabla^G H_s(Z)\big).
\end{align*}
Adding $\nabla_x Z_x+\nabla_y Z_y$ to both sides yields 
\begin{equation}\tau(Z) = \nabla_s\big(\nabla^G H_s(Z)\big)- J(Z)\nabla_x\big(\nabla^G H_s(Z)\big)
- K(Z)\nabla_y\big(\nabla^G H_s(Z)\big).
\end{equation}
Now, notice that since $\nabla J=\nabla K=0$, we have 
$$J(Z) \nabla_x + K(Z) \nabla_y = \nabla_{J(Z) Z_x + K(Z) Z_y}$$
and hence, using again the Fueter/Floer equation \eqref{eq:FueterFloerBG}, 
\begin{align*}
    \tau(Z) &= \nabla_s\big(\nabla^G H_s(Z)\big)-\nabla_{J(Z) Z_x + K(Z) Z_y} \big(\nabla^G H_s(Z)\big)\\
    &= \nabla^2 H_s(Z)[Z_s] + \nabla^G(\del_s H_s)(Z) - \nabla^2 H_s(Z)[J(Z) Z_x + K(Z) Z_y] \\
    & \ \ \ -J(Z) \nabla^G(\del_x H_s)(Z) - K(Z) \nabla^G(\del_yH_s)(Z) \\
    &= 2\nabla^2H_s(Z)[Z_s] - \nabla^2 H_s(Z)[\nabla^GH(Z)] \nonumber \\ & \ \ \ + \nabla^G(\del_s H_s)(Z) - J(Z) \nabla^G(\del_x H_s)(Z) - K(Z) \nabla^G (\del_yH_s)(Z),
\end{align*}
as claimed.
\end{proof}

The estimate in Lemma~\ref{lem:tension} shows that the tension field behaves, at least formally, as in the flat case: it is controlled pointwise by the $C^2$-norm of the Hamiltonian and the size of $Z_s$. In particular, the contribution of the tension term in \eqref{eq:Jost} can be handled as in the proof of Theorem~\ref{thm:C0boundflat}.
The main difficulty in the non-flat setting lies instead in the first term of \eqref{eq:Jost}, namely
\[
\mathrm{tr}\big(\nabla^2 H^0[\mathrm d Z,\mathrm d Z]\big).
\]
Unlike in the flat case, the Hessian of $H^0(q,p)=\tfrac12 |p|_g^2$ with respect to the Biquard--Gauduchon metric is not non-negative in general. Explicit computations in the model case of the hyperbolic plane show that its horizontal component carries a negative sign. However, this horizontal contribution is multiplied by a factor of order $|p|_g^2$. Consequently, close to the zero-section one is in a perturbative regime, where the negative part of the Hessian is quadratic in $|p|_g$ and therefore small. A more systematic analysis of this Hessian, possibly exploiting the locally symmetric structure of the base and the block-diagonal nature of the metric in the horizontal/vertical splitting, together with the fact that the Fueter/Floer equation \eqref{eq:FueterFloerBG} provides an $L^2$-control of the horizontal component of $\mathrm d Z$ in terms of the vertical component and the Hamiltonian, should lead to the desired quantitative confinement estimates in the non-flat setting. We leave this problem for future investigation.

\begin{remark}
In analogy with \cite[Section~7]{CRPS}, one expects that all null-homotopic critical points of the action functional $\mathcal A_H$ are contained in a compact subset of $\mathbb D^*_{\delta_0} Q$, whose size depends only on $\|h\|_{C^2}$. 
By introducing a suitable cut-off function, one may therefore assume, without loss of generality, that $h \equiv 0$ outside a compact subset of $\mathbb D^*_{\delta_0} Q$.
\end{remark}

\bibliography{mybib}{}

@article{Asselle,
  title={A note on the Morse homology for a class of functionals in Banach spaces involving the $2p$-area functional},
  author={Asselle, Luca and Starostka, Maciej},
  journal={NoDEA Nonlinear Differential Equations Appl.},
  volume={31},
  number={75},
  pages={109--160},
  year={2024},
  note={https://doi.org/10.1007/s00030-024-00962-3}
}

@article{Asselle2,
  title={Morse homology for a class of elliptic partial differential eqautions},
  author={Asselle, Luca and Cingolani, Silvia and Starostka, Maciej},
  journal={Commun. Contemp. Math.},
  year={2025},
  note={https://doi.org/10.1142/S0219199726500082}
}

@article{feix2001hyperkahler,
  title={Hyperk{\"a}hler metrics on cotangent bundles},
  author={Feix, Birte},
  journal={Journal fur die Reine und Angewandte Mathematik},
  pages={33--46},
  year={2001},
  publisher={WALTER DE GRUYTER}
}

@article{kaledin1997hyperkaehler,
  title={Hyperkaehler structures on total spaces of holomorphic cotangent bundles},
  author={Kaledin, Dmitry},
  journal={arXiv preprint alg-geom/9710026},
  year={1997}
}

@incollection{biquard2021hyperkahler,
  title={{Hyperk{\"a}hler metrics on cotangent bundles of Hermitian symmetric spaces}},
  author={Biquard, Olivier and Gauduchon, Paul},
  booktitle={Geometry and physics},
  pages={287--298},
  year={2021},
  publisher={CRC Press}
}

@article{eells,
  title={Harmonic mappings on Riemannian manifolds},
  author={Eells, James Jr. and Sampson, J. H.},
  journal={Amer. J. Math.},
  volume={86},
  number={1},
  pages={109--160},
  year={1964}
}

@article{CRPS,
  title={Generalizing symplectic topology from 1 to 2 dimensions},
  author={Brilleslijper, Ronen and Fabert, Oliver},
  journal={arXiv preprint arXiv:2412.16223},
  year={2024}
}

@book{jost2005riemannian,
  title={Riemannian geometry and geometric analysis},
  author={Jost, J{\"u}rgen},
  year={2005},
  publisher={Springer}
}

@article{albers2016cuplength,
  title={{Cuplength estimates in Morse cohomology}},
  author={Albers, Peter and Hein, Doris},
  journal={Journal of Topology and Analysis},
  volume={8},
  number={02},
  pages={243--272},
  year={2016},
  publisher={World Scientific}
}

@article{walpuski2017compactness,
  title={{A compactness theorem for Fueter sections}},
  author={Walpuski, Thomas},
  journal={Commentarii Mathematici Helvetici},
  volume={92},
  number={4},
  pages={751--776},
  year={2017}
}

@article{sacks1981existence,
  title={The existence of minimal immersions of 2-spheres},
  author={Sacks, Jonathan and Uhlenbeck, Karen},
  journal={Annals of mathematics},
  volume={113},
  number={1},
  pages={1--24},
  year={1981},
  publisher={JSTOR}
}

@article{HNS,
  title={{Hypercontact structures and Floer homology}},
  author={Hohloch, Sonja and Noetzel, Gregor and Salamon, Dietmar A},
  journal={Geometry \& Topology},
  volume={13},
  number={5},
  pages={2543--2617},
  year={2009},
  publisher={Mathematical Sciences Publishers}
}

@book{GT,
  title={Elliptic partial differential equations of second order},
  author={Giblarg, David and Trudinger, Neil S.},
  volume={224},
  year={2001},
  publisher={Springer}
}

@book{E,
  title={Partial differential equations},
  author={Evans, Lawrence C.},
  volume={19},
  year={1997},
  publisher={Graduate Studies in Mathematics, American Mathematical Society}
}

@article{beldjilali2023class,
  title={{On a class of K{\"a}hlerian manifolds.}},
  author={Beldjilali, Gherici},
  journal={International Journal of Open Problems in Computer Science \& Mathematics},
  volume={16},
  number={1},
  year={2023}
}

@article{esfahani2023towards,
  title={{Towards a monopole Fueter Floer homology I: a compactness theorem}},
  author={Esfahani, Saman Habibi},
  journal={arXiv preprint arXiv:2305.09456},
  year={2023}
}

@article{parker1996bubble,
  title={Bubble tree convergence for harmonic maps},
  author={Parker, Thomas H},
  journal={Journal of Differential Geometry},
  volume={44},
  number={3},
  pages={595--633},
  year={1996},
  publisher={Lehigh University}
}

@book{deLeonField,
  title={Methods of differential geometry in classical field theories: k-symplectic and k-cosymplectic approaches},
  author={De Le{\'o}n, Manuel and Salgado, Modesto and Vilarino, Silvia},
  year={2015},
  publisher={World Scientific}
}

@article{gunther1987polysymplectic,
  title={{The polysymplectic Hamiltonian formalism in field theory and calculus of variations. I. The local case}},
  author={G{\"u}nther, Christian},
  journal={Journal of differential geometry},
  volume={25},
  number={1},
  pages={23--53},
  year={1987},
  publisher={Lehigh University}
}

@article{CRMS,
      title={Floer sections in multisymplectic geometry}, 
      author={Ronen Brilleslijper and Oliver Fabert},
      journal = {International Journal of Geometric Methods in Modern Physics},
      year = {2026},
      doi = {10.1142/S0219887826501446},
      note={Accepted for publication. \url{https://doi.org/10.1142/S0219887826501446}}
}
\bibliographystyle{alpha}

\end{document}